\documentclass[a4paper,12pt]{article}

\usepackage{graphics}
\usepackage{epsf}

\usepackage{amsfonts}
\usepackage{amssymb}
\usepackage{epsfig}
\usepackage[latin1]{inputenc}
\usepackage{color}

\usepackage{amsthm,amsmath,amssymb}

\newtheorem{theorem}{Theorem}

\newtheorem{remark}[theorem]{Remark}

\begin{document}


\title{ SIMPLIFIED EXPLICIT EXPONENTIAL RUNGE-KUTTA METHODS WITHOUT ORDER REDUCTION}

\author{Bego\~na Cano
\thanks{Departamento de  Matem\'atica Aplicada, IMUVA, Universidad de Valladolid,  Spain \\ Email: bcano@uva.es}
\and
Mar\'\i a Jes\'us Moreta
\thanks{Departamento de An\'alisis Econ\'omico y Econom\'\i a Cuantitativa,  IMUVA, Universidad Complutense de Madrid, Spain\\ Email: mjesusmoreta@ccee.ucm.es}}

\date{}

\maketitle

\begin{abstract}
In a previous paper, a technique was suggested to avoid order reduction with any explicit exponential Runge-Kutta method when integrating initial boundary value nonlinear problems with time-dependent boundary conditions. In this paper, we significantly simplify the full discretization formulas to be applied under conditions which are nearly always satisfied in practice. Not only a simpler linear combination of $\varphi_j$-functions is given for both the stages and the solution, but also the information required on the boundary is so much simplified that, in order to get local order three, it is no longer necessary to resort to numerical differentiation in space. The technique is then shown to be computationally competitive against other widely used methods with high enough stiff order through the standard method of lines.

\end{abstract}



\section{Introduction}
Exponential methods have become valuable in the last two decades due to the recent development of Krylov methods to calculate exponential functions of matrices applied over vectors in reasonable computational time \cite{GG,GRT,HO2,HL97,ML}. Integrating in time reaction-diffusion problems  with standard Runge-Kutta (RK) methods  is well-known to require an implicit integration so that stability is observed without imposing impractical CFL conditions. However, exponential Runge-Kutta methods manage to integrate those problems in an \lq explicit' way, just at the cost of calculating iteratively exponential functions of matrices over vectors. It is justified in the literature \cite{T} that there are situations where it is cheaper to calculate the latter than solving the linear systems in the stages of implicit standard RK methods. This corresponds to the case where a good preconditioner is not available to solve efficiently those linear systems.

On the other hand, when using the standard method of lines to integrate initial boundary value problems, order reduction in time turns up with both standard and exponential RK methods. For the latter, stiff order conditions are given in \cite{LO0,HO} which allow to construct methods which avoid that order reduction. However, those restrictions may imply that the computational cost is bigger to get a desired accuracy; even that the required number of stages to get a  given stiff order of accuracy increases with respect to that necessary for the classical one. In contrast, the technique which is suggested in \cite{cm_erk} allows to avoid order reduction with any given explicit exponential RK method (EERK) just by adding in the formulas some information of the problem at the boundary. Since the number of nodes at the boundary is negligible compared to the total number of nodes, the computational cost of adding those terms is also negligible. Furthermore, as it will be seen in the numerical experiments, adding those terms may imply that Krylov subroutines converge more quickly (see \cite{CR1} for an explanation for that).

However, there are still a couple of items which can be improved in the technique which is suggested in \cite{cm_erk}.

On the one hand, although analytic expressions are given for the data of the problem (i.e. the boundary conditions and the source term and their derivatives), numerical differentiation in space may be necessary to approximate the required information on the boundary to get local order $3$. With Robin/Neumann boundary conditions, even numerical differentiation in time may be required. Furthermore, to get local order $4$, numerical differentiation in time and space may be necessary with both Dirichlet and Robin/Neumann boundary conditions. Although using numerical differentiation at the nodes on the boundary is not expensive, it would be better not to use it so as to avoid possible instabilities when the space grid or the time stepsize are very small.

On the other hand, using space numerical differentiation to calculate those boundary values implies that a CFL condition has been necessary in \cite{cm_erk} to prove the required order of convergence. Fortunately, this CFL condition is much weaker than that required to prove convergence with an explicit RK method, but it would also be better not having to impose it.

In this paper we will justify that, under some conditions on the coefficients of the method, which most EERK methods satisfy, some simplifications on the required boundary values to avoid order reduction can be performed. Thanks to them, no numerical differentiation in space will be required to achieve local order $3$ with both Dirichlet and Robin/Neumann boundary conditions, so that the weak CFL condition is no longer required in such a case. Moreover, simplified expressions to get local order $4$ are given and, although numerical differentiation is not completely avoided, it is not required for the calculation of the stages.

The paper is then structured as follows. The first section gives some preliminaries. The second one justifies the simplification of the full discretization formulas under certain assumptions. The third one proves that some of the simplifying assumptions are mostly always satisfied. Finally, some numerical experiments are given with both Dirichlet and Robin/Neumann boundary conditions, which confirm that using the suggested technique is cheaper than using the standard method of lines with other EERK methods which have been especially constructed to get the desired stiff order.

\section{Preliminaries}
\label{Preliminares}

In this paper, we are interested in integrating an initial boundary value problem of this type
\begin{eqnarray}
\label{laibvp}
\begin{array}{rcl}
u'(t)&=&Au(t)+f(t,u(t)), \quad  0\le t \le T,\\
u(0)&=&u_0 \in X,\\
\partial u(t)&=&g(t)\in Y, \quad  0\le t \le T,
\end{array}
\end{eqnarray}
under the assumptions stated in \cite{cm_erk}. For the time integration, explicit exponential Runge-Kutta integrators will be considered, which, when applied to an ODE differential system like this
\begin{eqnarray}
\dot{U}(t)&=& B U(T)+F(t,U(t)), \nonumber \\
 U(t_0)&=&U_0, \nonumber
\end{eqnarray}
where $B$ is a matrix, are defined by the following formulas
\begin{eqnarray}
K_{n,i}&=&e^{c_i k B}U_n+k \sum_{j=1}^{i-1} a_{ij}(k B) F(t_n+c_j k, K_{n,j}), \quad i=1,\dots,s, \label{etapas} \\
U_{n+1}&=&e^{k B}U_n+k \sum_{i=1}^s b_i(k B)F(t_n+c_i k,K_{n,i}),
\label{eerk}
\end{eqnarray}
where $k$ is the timestepsize and $U_n$ the approximation to $U(t_0+n k)$. Moreover, most methods of this type satisfy
\begin{eqnarray}
a_{ij}(z)&=&\sum_{l=1}^s \lambda_{i,j,l} \varphi_l( c_i z), \nonumber \\
b_i(z)&=& \sum_{l=1}^s \mu_{i,l} \varphi_l(z), \label{coefs}
\end{eqnarray}
where
$$
\varphi_l(z)=\int_0^1 e^{(1-\theta)z}\frac{\theta^{l-1}}{(l-1)!}d \theta, \quad l\ge 1,$$
 and $\lambda_{i,j,l}$ and $\mu_{i,l}$ are certain constants. (For some methods, it is allowed that the arguments of $\varphi_l$ in $a_{ij}$ depend on $c_r z$ for some $c_r \neq c_i$, which case is considered in \cite{cm_erk}, but in this paper we assume that the argument is always $c_i z$, so that the expressions for the formulas which manage to avoid order reduction are much simplified.) We will also assume that $\sum_{j=1}^s a_{ij}(0)=c_i$ (a standard simplifying assumption for Runge-Kutta methods), which implies that
\begin{eqnarray}
\sum_{j,l=1}^s \frac{\lambda_{i,j,l}}{l!}=c_i. \label{lambda}
\end{eqnarray}

For the space discretization, we consider methods satisfying the quite general assumptions in \cite{cm_erk}. In particular, we assume that, in order to integrate the elliptic problem
$$A u= f, \quad \partial u =g,$$
we use a certain matrix $A_{h,0}$ and certain boundary operators $C_h$ and $D_h$, so that the elliptic projection $U_h$ which approximates the solution of this problem is
\begin{eqnarray}
A_{h,0} U_h +C_h g= P_h f +D_h \partial f,
\label{sd}
\end{eqnarray}
where $P_h$ is a certain projection on nodal values.

In order to avoid the order reduction which turns up when considering the standard method of lines to integrate (\ref{laibvp}),  a modification is suggested in \cite{cm_erk}. The latter is based on integrating (\ref{laibvp}) firstly in time with the EERK method. For that, the exponential-type functions are substituted by appropriate initial boundary value problems for which suitable boundaries must be suggested. Then, the space discretization of those intermediate problems is performed.

If local order $p+1$ is searched for and the EERK method has classical order $\ge p$, the final formulas after the full discretization considering this modification consist of the following for $p=1,2,3$:

\subsection{$p=1$}
\begin{eqnarray}
K_{n,i,h}&=&e^{c_i k A_{h,0}}U_h^n+c_i k \varphi_1(c_i k A_{h,0})C_h \partial u(t_n)+k \sum_{j=1}^{i-1} \sum_{l=1}^s \lambda_{i,j,l}\varphi_l(c_i k A_{h,0})F_{n,j,h}, \nonumber \\
U_h^{n+1}&=&e^{k A_{h,0}}U_h^n+k \varphi_1(K A_{h,0})[C_h \partial u(t_n)-D_h \partial Au(t_n)]+k^2 \varphi_2(K A_{h,0})C_h \partial A u(t_n) \nonumber \\
&&+k \sum_{i=1}^s \sum_{l=1}^s \mu_{i,l} [\varphi_l(k A_{h,0})F_{n,i,h}+k \varphi_{l+1} (k A_{h,0}) C_h \partial f(t_n,u(t_n))], \label{for1}
\end{eqnarray}
where
$$F_{n,j,h}=f(t_n+c_j k, K_{n,j,h}).$$
\begin{remark}
As stated in \cite{cm_erk}, the required expresions on the boundary in (\ref{for1}) can  be exactly calculated in terms of data with Dirichlet boundary conditions since, in such a case,
$$\partial u(t_n)=g(t_n), \quad \partial f(t_n,u(t_n))=f(t_n,g(t_n)), \quad \partial A u(t_n)=\dot{g}(t_n)-f(t_n,g(t_n)).$$
As for Robin/Neumann boundary conditions, again $\partial u(t_n)=g(t_n)$, but $\partial f(t_n,u(t_n))$ must be approximated through the numerical approximation that the proper space discretization  (\ref{sd}) gives for $u(t_n)$ on the boundary. That does not require numerical differentiation to approximate those values and thus no additional type of errors in the global error of the full discretization is added except for that coming from the chosen space discretization (\ref{sd}) itself.
\end{remark}
\subsection{$p=2$}
\begin{eqnarray}
K_{n,i,h}&=&e^{c_i k A_{h,0}}U_h^n+c_i k \varphi_1(c_i k A_{h,0})[C_h \partial u(t_n)-D_h \partial A u(t_n)]
\nonumber \\
&&+(c_i k)^2 \varphi_2(c_i k A_{h,0})C_h \partial A u(t_n) \nonumber \\
&&+k \sum_{j=1}^{i-1} \sum_{l=1}^s \lambda_{i,j,l}[\varphi_l(c_i k A_{h,0})F_{n,j,h}+c_i k \varphi_{l+1}(c_i k A_{h,0})C_h \partial f(t_n,u(t_n))], \nonumber \\
U_h^{n+1}&=&e^{k A_{h,0}}U_h^n+k \varphi_1(k A_{h,0})[C_h \partial u(t_n)-D_h \partial Au(t_n)]
\nonumber \\
&&+k^2 \varphi_2(k A_{h,0})[C_h \partial A u(t_n)-D_h \partial A^2 u(t_n)]+k^3 \varphi_3(k A_{h,0})C_h \partial A^2 u(t_n) \nonumber \\
&&+k \sum_{i=1}^s \sum_{l=1}^s \mu_{i,l}\big[\varphi_l(k A_{h,0})F_{n,i,h} \nonumber \\
&&\hspace{2.5cm}+k \varphi_{l+1} (k A_{h,0}) [C_h \partial f_{n,i,1}-D_h \partial A f(t_n,u(t_n))]\nonumber \\
&&\hspace{2.5cm}+k^2 \varphi_{l+2}(k A_{h,0})C_h \partial A f(t_n,u(t_n))\big], \label{for2}
\end{eqnarray}
where
$$
f_{n,i,1}=f(t_n+c_i k,u(t_n)+c_i k \dot{u}(t_n)).$$
Using similar arguments as in \cite{cm_erk}, this expression can be substituted by this other one without losing local order $3$ since the difference between both expressions is $O(k^2)$:
\begin{eqnarray}
f_{n,i,1}\approx  f(t_n,u(t_n))+c_i k[f_t(t_n,u(t_n))+f_u(t_n,u(t_n))\dot{u}(t_n)].
\label{fni1}
\end{eqnarray}
\begin{remark}
In general,  the expressions on the boundary in (\ref{for2}) can just be calculated now in an approximated way even with Dirichlet boundary conditions since
\begin{eqnarray}
\partial A^2 u= \partial [\ddot{u}-f_t -f_u \dot{u}- A f],
\label{a2u}
\end{eqnarray}
and, in order to approximate $\partial A f(t_n,u(t_n))$, numerical differentiation in space is required. As for Robin/Neumann boundary conditions, $\dot{u}(t_n)$ at the boundary must also be approximated through the numerical values given by the space discretization of the problem at the boundary. More precisely, numerical differentiation in time is also required. All this leads to some additional terms on the global error depending on the particular methods of numerical differentiation being used. But what is more important, in order to recover the classical order of the method, a CFL like
\begin{eqnarray}
|\frac{k}{h^\gamma}|\le C,
\label{CFL}
\end{eqnarray}
is required, where $\gamma$ is the order of the space derivative to be approximated. This condition is much weaker than the one which would be required with an explicit RK method because the order of the space derivatives to be approximated are smaller than the order of the space derivatives in $A$. (Notice, for example, that if $A$ is the second derivative in space in $1$ dimension and $f=f(u)$,
 $$ Af(u)=f''(u)u_x^2+f'(u)u_{xx}=f''(u)u_x^2+f'(u)(\dot{u}-f(u)),$$
 and just $u_x$ must be numerically approximated when using Dirichlet boundary conditions. With a Neumann boundary condition,
 as
 $$
 \partial (A f(u))=f'''(u) u_x^3+3 f''(u)u_x(\dot{u}-f(u))+f'(u)(\dot{u}_x-f'(u)u_x),$$
 no numerical differentation in space is required in this example in fact.) In any case, it would be even better if a condition like (\ref{CFL}) did not turn up with any problem.
\end{remark}

\subsection{$p=3$}
\begin{eqnarray}
K_{n,i,h}&=&e^{c_i k A_{h,0}}U_h^n+c_i k \varphi_1(c_i k A_{h,0})[C_h \partial u(t_n)-D_h \partial A u(t_n)]
\nonumber \\
&&+(c_i k)^2 \varphi_2(c_i k A_{h,0})[C_h \partial A u(t_n)-D_h \partial A^2 u(t_n)] \nonumber \\
&&+(c_i k)^3 \varphi_3(c_i k A_{h,0})C_h \partial A^2 u(t_n) \nonumber \\
&&+k \sum_{j=1}^{i-1} \sum_{l=1}^s \lambda_{i,j,l}\big[\varphi_l(c_i k A_{h,0})F_{n,j,h}
\nonumber \\
&&\hspace{2.5cm}+c_i k \varphi_{l+1}(c_i k A_{h,0})[C_h \partial f_{n,j,1}-D_h \partial A f(t_n,u(t_n))] \nonumber \\
&&\hspace{2.5cm}+(c_i k)^2 \varphi_{l+2}(c_i k A_{h,0})C_h \partial A f(t_n,u(t_n))\big] \nonumber \\
U_h^{n+1}&=&e^{k A_{h,0}}U_h^n+k \varphi_1(k A_{h,0})[C_h \partial u(t_n)-D_h \partial Au(t_n)]
\nonumber \\
&&+k^2 \varphi_2(k A_{h,0})[C_h \partial A u(t_n)-D_h \partial A^2 u(t_n)] \nonumber \\
&&+k^3 \varphi_3(k A_{h,0})[C_h \partial A^2 u(t_n)-D_h \partial A^3 u(t_n)] \nonumber \\
&&+k^4 \varphi_4(k A_{h,0}) C_h \partial A^3 u(t_n) \nonumber \\
&&+k \sum_{i=1}^s \sum_{l=1}^s \mu_{i,l}\big[\varphi_l(k A_{h,0})F_{n,i,h} \nonumber \\
&&\hspace{2.5cm}+k \varphi_{l+1} (k A_{h,0}) [C_h \partial f_{n,i,2}-D_h \partial A f_{n,i,1}]\nonumber \\
&&\hspace{2.5cm}+k^2 \varphi_{l+2}(k A_{h,0})[C_h \partial A f_{n,i,1}-D_h \partial A^2 f(t_n,u(t_n))] \nonumber \\
&&\hspace{2.5cm}+k^3 \varphi_{l+3}(k A_{h,0})C_h \partial A^2 f(t_n,u(t_n))
\big], \label{for3}
\end{eqnarray}
where
\begin{eqnarray}
f_{n,i,2}&=&f\bigg(t_n+c_i k, u(t_n)+c_i k A u(t_n)+\frac{(c_i k)^2}{2} A^2 u(t_n) \nonumber \\
&&+k \sum_{j,l} \lambda_{i,j,l} \big[ \frac{1}{l!} f(t_n+c_j k, u(t_n)+c_j k \dot{u}(t_n))+ \frac{c_i k}{(l+1)!} A f(t_n,u(t_n))\big]\bigg). \nonumber
\end{eqnarray}
Again, without losing local order 4, and with similar arguments to those in \cite{cm_erk},  $f_{n,i,2}$ can be  substituted by the following expression which differs from the previous one in $O(k^3)$ and for which (\ref{lambda}) and (\ref{laibvp}) have been used. For the sake of brevity, we assume that all terms are evaluated at $t_n$ or $(t_n,u(t_n))$.
\begin{eqnarray}
f_{n,i,2}&\approx & f+c_i k[f_t+f_u\dot{u}]+\frac{c_i k^2}{2} f_u \big[c_i(\ddot{u}-f_t-f_u\dot{u})+(2 \sum_{j,l} \frac{\lambda_{i,j,l}}{(l+1)!}-c_i) Af\big] \nonumber \\
&&+k^2 (\sum_{j,l=1}^s \frac{\lambda_{i,j,l}}{l!}c_j ) f_u(f_t+f_u\dot{u})+\frac{(c_i k)^2}{2}\big[f_{tt}+2 f_{tu}\dot{u}+f_{uu}[\dot{u},\dot{u}]\big].\nonumber
\end{eqnarray}
\begin{remark}
Notice that,  from (\ref{laibvp}),
$$
A^3 u=[\stackrel{\dots}{u}-f_{tt}-2 f_{tu} \dot{u}-f_{uu}[\dot{u},\dot{u}]-f_u \ddot{u}-A f_t-A (f_u \dot{u})-A^2 f].
$$
Then, with Dirichlet boundary conditions, in order to calculate the boundary of the last two terms in the previous expression, it is necessary to calculate a certain derivative in space of the exact solution at the boundary, and also the derivative with respect to time of that space derivative. This leads to the fact that, apart from the CFL condition (\ref{CFL}) in order to prove convergence, a term in the global error coming from both numerical differentiation in space and time turns up (see \cite{cm_erk}).

As for Robin/Neumann boundary conditions, because of the fifth term above, it is also necessary in general to resort to numerical differentiation to approximate a second derivative in time.

\end{remark}

\section{Simplification of full discretization formulas under certain assumptions}

In this section we will see that, under certain conditions on the coefficients of the EERK method (\ref{coefs}), the above formulas (\ref{for1}),(\ref{for2}),(\ref{for3}) can be significantly simplified and what is more important, the required boundary values can be calculated in a much easier way eliminating in some cases the necessity to resort to numerical differentiation to approximate those values.

The assumptions on (\ref{coefs}) which we will consider for all $p=1,2,3$ are the following
\begin{eqnarray}
\sum_{i=1}^s \mu_{i,1}=1, \quad \sum_{i=1}^s \mu_{i,l}=0 \quad (l=2,\dots,s). \label{cond_mu}
\end{eqnarray}
We will also write the obtained formulas as the corresponding linear combination of $\varphi_l$-functions, so that existing Krylov subroutines can be directly applied \cite{nw}. Moreover, we will make an effort to write the necessary values on the boundary in the more explicit way we can in terms of data. (Again, all the terms on the boundary are assumed to be evaluated at $t_n$ or $(t_n,u(t_n))$.)

\subsection{$p=1$}

We notice that, using (\ref{cond_mu}) for $U_h^{n+1}$, (\ref{for1}) can be written as
\begin{eqnarray}
K_{n,i,h}&=&e^{c_i k A_{h,0}}U_h^n+ k \varphi_1(c_i k A_{h,0})[\sum_{j=1}^{i-1} \lambda_{i,j,1} F_{n,j,h}+c_i C_h \partial u] \nonumber \\
&&+k \sum_{l=2}^s \varphi_l(c_i k A_{h,0})\sum_{j=1}^{i-1}  \lambda_{i,j,l}F_{n,j,h}, \nonumber \\
U_h^{n+1}&=&e^{k A_{h,0}}U_h^n+k \varphi_1(k A_{h,0})\big[\sum_{i=1}^s \mu_{i,1} F_{n,i,h}+ C_h \partial u-D_h \partial[\dot{u}-f] \big] \nonumber \\
&&+k \varphi_2(k A_{h,0})[\sum_{i=1}^s  \mu_{i,2}F_{n,i,h}+k C_h \partial\dot{u}] \nonumber \\
&&+k  \sum_{l=3}^s  \varphi_l(k A_{h,0})\sum_{i=1}^s \mu_{i,l} F_{n,i,h}, \nonumber
\end{eqnarray}

\begin{remark}
We notice that here we have used the first condition in (\ref{cond_mu}) to simplify the expression on the boundary multiplying $\varphi_2(k A_{h,0})$ in the formula for $U_h^{n+1}$. In such a way, that can be exactly calculated in terms of data with both Dirichlet and Robin/Neumann boundary conditions. On the other hand, the second conditions in (\ref{cond_mu}) allow to annilihate the terms on the boundary multiplying $\varphi_l(k A_{h,0})C_h$ for $l\ge 3$.
\label{rem1}
\end{remark}

\subsection{$p=2$}
In this case, we will also assume that the coefficients $\lambda_{i,j,l}$ in (\ref{coefs}) satisfy
\begin{eqnarray}
\sum_{j=1}^{i-1} \lambda_{i,j,1}=c_i, \quad \sum_{j=1}^{i-1} \lambda_{i,j,l}=0 \quad (l=2,\dots,s), \quad i=1,\dots,s. \label{cond_lambda}
\end{eqnarray}
This is not  necessary to calculate the required boundary values in the stages without resorting to numerical differentiation, but it allows to simplify the expressions considerably.
More precisely, (\ref{for2}) reduces to
\begin{eqnarray}
K_{n,i,h}&=&e^{c_i k A_{h,0}}U_h^n+ k \varphi_1(c_i k A_{h,0})\big[\sum_{j=1}^{i-1} \lambda_{i,j,1} F_{n,j,h}+c_i [C_h \partial u-D_h \partial[\dot{u}-f]\big] \nonumber \\
&&+k \varphi_2(c_i k A_{h,0})[\sum_{j=1}^{i-1}\lambda_{i,j,2} F_{n,j,h}+c_i^2 k C_h \partial \dot{u}]  \nonumber \\
&&+k \sum_{l=3}^s \varphi_l(c_i k A_{h,0})\sum_{j=1}^{i-1}  \lambda_{i,j,l} F_{n,j,h}, \nonumber \\
U_h^{n+1}&=&e^{k A_{h,0}}U_h^n+k \varphi_1(k A_{h,0})\big[\sum_{l=1}^s \mu_{i,1} F_{n,i,h}+ C_h \partial u-D_h \partial[\dot{u}-f] \big] \nonumber \\
&&+k \varphi_2(k A_{h,0})\bigg[\sum_{i=1}^s  \mu_{i,2}F_{n,i,h}+k C_h \partial[\dot{u}+k(\sum_{i=1}^s \mu_{i,1} c_i)(f_t+f_u \dot{u})] \nonumber \\
&&\hspace{2cm}-k D_h \partial[\ddot{u}-f_t-f_u\dot{u}]\bigg] \nonumber \\
&&+k \varphi_3(k A_{h,0})\bigg[\sum_{i=1}^s  \mu_{i,3}F_{n,i,h}+k^2 C_h \partial[\ddot{u}+(\sum_{i=1}^s \mu_{i,2} c_i-1)(f_t+f_u\dot{u})]\bigg] \nonumber \\
&&+k  \sum_{l=4}^{s+1}  \varphi_l(k A_{h,0})\bigg[\sum_{i=1}^s \mu_{i,l} F_{n,i,h}+k^2 (\sum_{i=1}^s \mu_{i,l-1} c_i)C_h \partial[f_t+f_u\dot{u}]\bigg], \label{fdu2}
\end{eqnarray}
where, in order to obtain the expression in $U_h^{n+1}$, (\ref{cond_mu}), (\ref{fni1}) and (\ref{a2u}) have been used. (It is understood that $\mu_{i,s+1}=0$.)
\begin{remark}
We can then see that, in contrast to (\ref{for2}), the required boundary values can be calculated exactly in terms of data with Dirichlet boundary conditions and using the approximation of the space discretization of (\ref{laibvp}) at the boundary together with numerical differentiation in time when considering Robin/Neumann boundary conditions. (For the latter, notice for example that, with Neumann boundary condition in one dimension,
$$
\partial f_u(t,u)\dot{u}=f_{u,x}(t,u)\dot{u}+f_{uu}(t,u)u_x\dot{u}+f_u(t,u)\dot{u}_x,
$$
and then $u_x=g$, $\dot{u}_x=\dot{g}$ and it is just $\dot{u}$ which must be approximated through numerical differentiation in time from the approximated values at the boundary.)
In any case, numerical differentiation in space is not necessary to approximate the required boundary values with either Dirichlet or Robin/Neumann boundary conditions. Because of this, under assumptions (\ref{cond_mu}), the term $\nu_h$ concerning the error from the numerical differentiation in space does not turn up in the expression for the global error when $p=2$ in \cite{cm_erk}. And what is more important, condition (\ref{CFL}) is not further required to prove convergence.
\label{rem2}
\end{remark}

\subsection{$p=3$}
\label{sect23}
For this value of $p$, numerical differentiation in space is required for the boundaries in the stages in (\ref{for3}) when considering Dirichlet boundary conditions. However, with the simplification (\ref{cond_lambda}), that is not required any more. As for Robin/Neumann boundary conditions, the numerical differentiation in space might have  been necessary in (\ref{for3}), but it is not either necessary under the same assumptions. More precisely, using (\ref{cond_lambda}) we have
\begin{eqnarray}
K_{n,i,h}&=&e^{c_i k A_{h,0}}U_h^n+k \varphi_1(k A_{h,0})\bigg[\sum_{j=1}^{i-1} \lambda_{i,j,1}F_{n,j,h}+c_i\big[C_h \partial u-D_h[\dot{u}-f]\big]\bigg] \nonumber \\
&&+k \varphi_2(c_i k A_{h,0})\bigg[ \sum_{j=1}^{i-1} \lambda_{i,j,2} F_{n,j,h}+c_i k C_h \partial [c_i \dot{u}+k\big(\sum_{j=1}^{i-1} \lambda_{i,j,1} c_j\big)(f_t+f_u \dot{u})]\nonumber \\
&& \hspace{2.5cm}-c_i^2 k D_h \partial[\ddot{u}-f_t-f_u\dot{u}]\bigg] \nonumber \\
&&+k \varphi_3(c_i k A_{h,0})\bigg[ \sum_{j=1}^{i-1} \lambda_{i,j,3} F_{n,j,h}+c_i k^2 C_h \partial [c_i^2 \ddot{u}+(\sum_{j=1}^{i-1} \lambda_{i,j,2} c_j -c_i^2)(f_t+f_u \dot{u})]\bigg]\nonumber \\
&&+k \sum_{l=4}^{s+1} \varphi_l(k A_{h,0})\bigg[\sum_{j=1}^{i-1} \lambda_{i,j,l} F_{n,j,h}+c_i k^2 \big(\sum_{j=1}^{i-1} \lambda_{i,j,l-1} c_j\big) C_h \partial [f_t+f_u\dot{u}]\bigg], \label{fdk3}
\end{eqnarray}
where $\lambda_{i,j,l}$ is understood to be zero for $l>s$. We notice that
just numerical differentiation in time with Robin/Neumann boundary conditions would be required.

As for $U_h^{n+1}$, using also (\ref{cond_mu}), the formulas in (\ref{for3}) can be simplified to the following:
\begin{eqnarray}
U_h^{n+1}&=&e^{k A_{h,0}}U_h^n\nonumber \\
&&+k \varphi_1(k A_{h,0})\big[ \sum_{i=1}^s \mu_{i,1} F_{n,i,h}+C_h \partial u-D_h \partial[\dot{u}-f]\big] \nonumber \\
&&+k \varphi_2(k A_{h,0})\big[ \sum_{i=1}^s \mu_{i,2} F_{n,i,h}+k C_h \partial b_{2,n,c}-k D_h \partial b_{2,n,d}\big] \nonumber \\
&&+k \varphi_3(k A_{h,0})\big[ \sum_{i=1}^s \mu_{i,3} F_{n,i,h}+k^2 C_h \partial b_{3,n,c}-k^2 D_h \partial b_{3,n,d}\big] \nonumber \\
&&+k \varphi_4(k A_{h,0})\big[ \sum_{i=1}^s \mu_{i,4} F_{n,i,h}+k^2 C_h \partial b_{4,n,c}-k^2 D_h \partial b_{4,n,d} \big] \nonumber \\
&&+k \sum_{l=5}^{s+2} \varphi_l(k A_{h,0})\big[ \sum_{i=1}^s \mu_{i,l} F_{n,i,h}+k^2 C_h \partial b_{l,n,c}-k^2 D_h \partial b_{l,n,d} \big]
 \label{fdu3}
\end{eqnarray}
where it is understood that $\mu_{i,l}=0$ for $l>s$ and the boundary terms $b_{l,n,c}$ and $b_{l,n,d}$ ($l\ge 2$) are given by
\begin{eqnarray}
b_{2,n,c}&=& \dot{u}+k\big( \sum_{l=1}^s \mu_{i,1} c_i) (f_t+f_u \dot{u}) +k^2 \sum_{i=1}^s \mu_{i,1} \big(\sum_{j,l} \lambda_{i,j,l} \frac{c_j}{l!}-\frac{c_i^2}{2}\big) f_u (f_t+f_u \dot{u})   \nonumber \\
&&+\frac{k^2}{2} \sum_{i=1}^s \mu_{i,1} c_i \big(f_u[c_i \ddot{u}+(2 \sum_{j,l}\frac{\lambda_{i,j,l}}{(l+1)!}-c_i) Af]
+c_i(f_{tt}+2 f_{tu} \dot{u}+f_{u,u}[\dot{u},\dot{u}])\big) \nonumber \\
b_{2,n,d}&=&\ddot{u}-f_t-f_u \dot{u}+k\big(\sum_{i=1}^s \mu_{i,1}c_i\big)A (f_t+f_u \dot{u}), \nonumber \\
b_{3,n,c}&=&(\sum_{i=1}^s \mu_{i,2}c_i)(f_t+f_u \dot{u})+\ddot{u}-f_t-f_u \dot{u} \nonumber \\
&&+\frac{k}{2}\sum_{i=1}^s \mu_{i,2}c_i \big( f_u[c_i \ddot{u}+2(\sum_{j,l}\frac{\lambda_{i,j,l}}{(l+1)!}-c_i)A f+c_i(
f_{tt}+2 f_t f_u \dot{u}+f_{uu}[\dot{u},\dot{u}]\big)\nonumber \\&&+k  \sum_{i=1}^s \mu_{i,2} \big(\sum_{j,l} \lambda_{i,j,l}\frac{c_j}{l!}-\frac{c_i^2}{2}\big) f_u[f_t+f_u \dot{u}]  +k \big( \sum_{i=1}^s \mu_{i,1} c_i \big) A(f_t+f_u \dot{u}),\nonumber \\
b_{3,n,d}&=&\stackrel{\dots}{u}-f_{tt}-2 f_{tu}\dot{u}-f_u \ddot{u}-f_{uu}[\dot{u},\dot{u}]+(\sum_{i=1}^s \mu_{i,2}c_i-1) A(f_t+f_u \dot{u}), \nonumber \\
b_{4,n,c}&=&(\sum_{i=1}^s \mu_{i,3}c_i)(f_t+f_u \dot{u})+k[\stackrel{\dots}{u}-f_{tt}-2 f_{tu}\dot{u}-f_u[\dot{u},\dot{u}]-f_u\ddot{u}-A(f_t+f_u \dot{u})] \nonumber \\
&&+\frac{k}{2}\sum_{i=1}^s \mu_{i,3}c_i\big[ f_u [c_i\ddot{u}+2(\sum_{j,l}\frac{\lambda_{i,j,l}}{(l+1)!}-c_i)Af]+c_i(f_{tt}+2 f_tf_u \dot{u}+f_{uu}[\dot{u},\dot{u}])\big] \nonumber \\
&&+k  \sum_{i=1}^s \mu_{i,3} \big( \sum_{j,l} \lambda_{i,j,l}\frac{c_j}{l!}-\frac{c_i^2}{2}\big) f_u[f_t+f_u \dot{u}]  \nonumber \\
&&+k \big( \sum_{i=1}^s \mu_{i,2} c_i \big) A(f_t+f_u \dot{u}), \nonumber \\
b_{4,n,d}&=&\big( \sum_{i=1}^s \mu_{i,3} c_i\big) A(f_t+f_u \dot{u}), \nonumber
\end{eqnarray}
and, for $l \ge 5$,
\begin{eqnarray}
b_{l,n,c}&=& (\sum_{i=1}^s \mu_{i,l-1}c_i)(f_t+f_u \dot{u})\nonumber \\
&&+\frac{k}{2}\sum_{i=1}^s \mu_{i,l-1}c_i\big(f_u[c_i \ddot{u}+(2 \sum \frac{\lambda_{i,j,l}}{(l+1)!}-c_i) Af]+c_i(f_{tt}+2 f_t f_u \dot{u}+f_{uu}[\dot{u},\dot{u}])\big)  \nonumber \\
&&+k  \sum_{i=1}^s \mu_{i,l-1} \big(\sum_{j,m} \lambda_{i,j,m}\frac{c_j}{m!}-\frac{c_i^2}{2}\big) f_u[f_t+f_u \dot{u}]  \nonumber \\
&&+k \big( \sum_{i=1}^s \mu_{i,l-2} c_i \big) A(f_t+f_u \dot{u}), \nonumber \\
b_{l,n,d}&=& (\sum_{i=1}^s \mu_{i,l-1}c_i)A(f_t+f_u \dot{u}).\nonumber
\end{eqnarray}

\begin{remark}
We notice that, with all Neumann, Robin and Dirichlet boundary conditions, calculating $\partial A(f_t+f_u \dot{u})$ requires in general numerical differentiation in space. Therefore, we need condition (\ref{CFL}) to prove convergence. Moreover, with Dirichlet boundary conditions we require numerical differentiation in time to approximate the first time derivative of some derivative in space and, with Robin/Neumann ones, also the second time derivative of the exact solution itself. No higher order derivatives are required in the rest of the terms.
\label{rem3}
\end{remark}

\section{Discussion on simplifying assumptions}

In this section we will basically prove that condition (\ref{cond_mu}) is satisfied whenever $s\le q$ if $q$ is the classical order of the method. As, looking for efficiency, a small number of stages is searched for, that is not a restriction in practice. On the other hand, although condition (\ref{cond_lambda}) does not come from any condition of consistency, all the most well-known EERK methods satisfy it.

Let us  precisely state the first result.
\begin{theorem}
Let us assume that the EERK method (\ref{eerk}) with coefficients (\ref{coefs}) has classical order $\ge q$ with $q=1,2,3,4$. Then, if $s\le q$, (\ref{cond_mu}) holds.
\label{th}
\end{theorem}
\begin{proof}
According to \cite{HO}, if the method has classical order $\ge q$, the following conditions must hold:
\begin{eqnarray}
&&q=1, \qquad \sum_{i=1}^s \sum_{l=1}^s \mu_{i,l} \frac{1}{l!}=1, \label{cp1} \\
&&q=2, \qquad \sum_{i=1}^s \sum_{l=1}^s \mu_{i,l} \frac{1}{(l+1)!}=\frac{1}{2}, \label{cp2} \\
&&q=3, \qquad \sum_{i=1}^s \sum_{l=1}^s \mu_{i,l} \frac{1}{(l+2)!}=\frac{1}{6}, \label{cp3} \\
&&q=4, \qquad \sum_{i=1}^s \sum_{l=1}^s \mu_{i,l} \frac{1}{(l+3)!}=\frac{1}{24}. \label{cp4}
\end{eqnarray}
If $s=1$, (\ref{cp1}) leads to $\mu_{1,1}=1$ and (\ref{cond_mu}) follows.  If $s=2$ and $q\ge 2$, (\ref{cp1}) and (\ref{cp2}) lead to a system of two equations in the two unknowns $\mu_{1,1}+\mu_{2,1}$ and $\mu_{1,2}+\mu_{2,2}$. The coefficient matrix of this system is
$$\left( \begin{array}{cc} 1 & 1/2 \\ 1/2 & 1/6 \end{array} \right).$$
Therefore, the  unique solution of this system is
$$ \mu_{1,1}+\mu_{2,1}=1, \quad \mu_{1,2}+\mu_{2,2}=0,$$
which corresponds to (\ref{cond_mu}). For $s=3$ and $s=4$ a similar argument follows. If $q\ge s$, either (\ref{cp1})-(\ref{cp3}) or (\ref{cp1})-(\ref{cp4}) lead to a system of $s$ equations in the $s$ unknowns
$$\sum_{i=1}^s \mu_{i,l}, \quad l=1,\dots,s.$$
The coefficient matrix are both regular and it happens that the first column of those matrices coincide with the source term in the linear system. Therefore, the solution is the first canonical vector, i.e. (\ref{cond_mu}).
\end{proof}

\section{Numerical experiments}

In this section, we have made some numerical experiments in order to corroborate that the full discretization formulas lead to methods which are efficient in two senses. On the one hand, order reduction is avoided, as theoretically justified in \cite{cm_erk}, and the CFL condition (\ref{CFL}), in principle necessary for $p=3$, does not seem very restrictive. On the other hand, as just information on the boundary must be added in order to avoid order reduction, the technique seems cheaper than the use of EERK methods with high stiff order, which requires that the coefficients of the method satisfies some restrictions and, due to that, the local error may be higher and also the number of stages to get a desired order of accuracy.

For the numerical experiments we have used standard Krylov subroutines \cite{nw} since they seem to be the most general for problems which may arise in practice. Of course, for the particular examples in this section, other cheaper FFT techniques could be used, but the aim is to foresee the comparison in efficiency in more complex problems.

\subsection{One-dimensional problem}

\begin{figure}
\vspace{-4cm}
\begin{center}
\includegraphics[width=130mm]{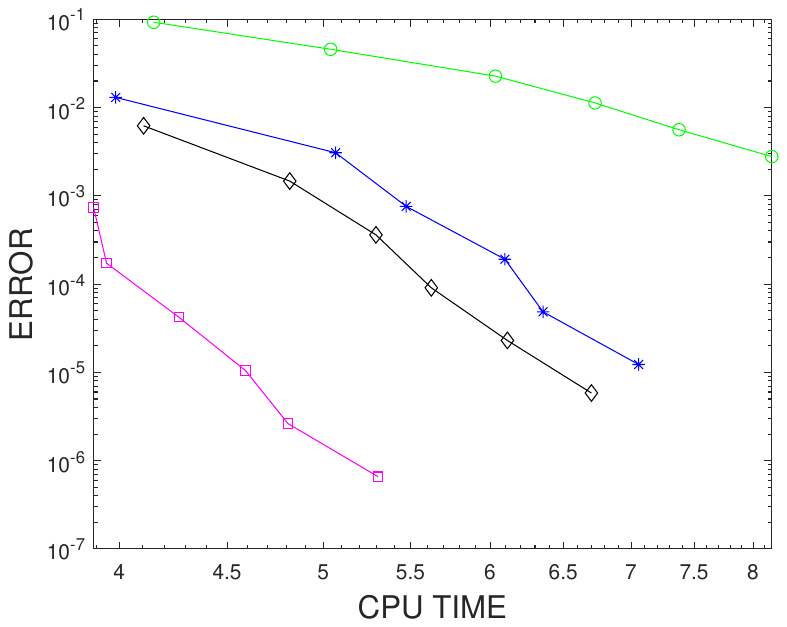}
\end{center}
\vspace{-5cm}
\caption{Error against CPU time when integrating problem (\ref{p1}) with nonvanishing boundary conditions (\ref{dir}), using (\ref{rk2}) without avoiding order reduction
(green circles), the suggested technique (\ref{rk2p1}) with $p=1$ (blue asterisks), the suggested technique  corresponding to $p=2$  (\ref{rk2p2}) (magenta squares) and the method with stiff order $2$ (\ref{rk2b}) using the standard method of lines (black diamonds).}
\label{f1}
\end{figure}

\begin{figure}
\vspace{-4cm}
\begin{center}
\includegraphics[width=130mm]{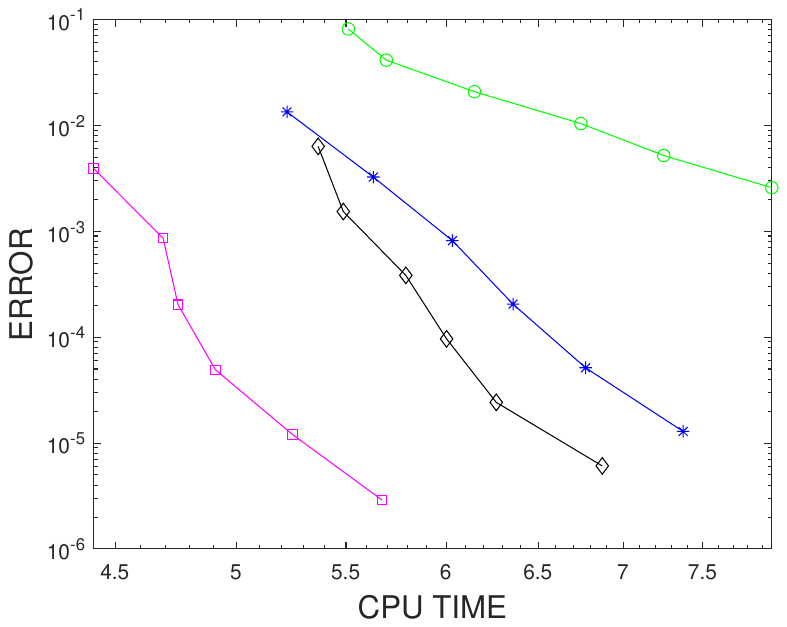}
\end{center}
\vspace{-5cm}
\caption{Error against CPU time when integrating problem (\ref{p1}) with nonvanishing boundary conditions (\ref{dn}), using (\ref{rk2}) without avoiding order reduction
(green circles), the suggested technique (\ref{rk2p1}) with $p=1$ (blue asterisks), the suggested technique  corresponding to $p=2$ (\ref{rk2p2}) (magenta squares) and the method with stiff order $2$ (\ref{rk2b}) using the standard method of lines (black diamonds).}
\label{f1dn}
\end{figure}

We have firstly considered the one-dimensional problem
\begin{eqnarray}
u_t=u_{xx}+u^2 +h(x,t), \quad x \in [0,1],\quad  t\in [0,1],
\label{p1}
\end{eqnarray}
with function $h$ and initial and boundary conditions so that the exact solution is $u(x,t)=\cos(x+t)$. In the first place, we consider  Dirichlet boundary conditions at both sides
\begin{eqnarray}
u(0,t)=g_0(t), \quad u(1,t)=g_1(t),
\label{dir}
\end{eqnarray}
and, secondly, Dirichlet at $x=0$ and Neumann at $x=1$,
\begin{eqnarray}
u(0,t)=g_0(t), \quad u_x(1,t)=g_1(t).
\label{dn}
\end{eqnarray}
As stated in \cite{cm_erk}, the assumptions there are then satisfied with $X=C[0,1]$ and $\|\cdot\|$ the maximum norm.

Then, for the space discretization, we have taken the standard second-order difference scheme, for which $D_h\equiv 0$ in (\ref{sd}) and,  in the case of two  Dirichlet boundary conditions,
$$A_{h,0}=\mbox{tridiag}(1,-2,1)/h^2, \quad C_h[g_0(t),g_1(t)]^T =[g_0(t),0,\dots,0,g_1(t)]^T/h^2,$$
while, in the Dirichlet and Neumann one, $A_{h,0}$ has one more row and column since the value at the final node is also numerically approximated. More precisely, by considering a centered difference in that last node to approximate the Neumann condition,
$$
A_{h,0}=\frac{1}{h^2} \left( \begin{array}{cccccc} -2 & 1 & 0 & \dots & \dots & 0 \\ 1 & -2 & 1 & 0 & \dots & \vdots \\ 0 & 1  & -2 & 1 & \ddots & \vdots \\
\vdots & 0 & \ddots & \ddots & \ddots & 0 \\ \vdots & & \ddots &1 & -2& 1 \\ 0 & \dots & \dots &0 &2 & -2 \end{array}\right), \quad C_h \left( \begin{array}{c} g_0(t) \\ g_1(t) \end{array} \right)=\left( \begin{array}{c} \frac{1}{h^2} g_0(t) \\ 0 \\ \vdots \\\frac{2}{h}g_1(t) \end{array} \right).
$$
For the time integration, we have taken the EERK method of classical order $2$
\begin{eqnarray}
\begin{array}{c|cc} 0 &  & \\ \frac{1}{2} & \frac{1}{2}\varphi_{1,2} &  \\  \hline & 0 & \varphi_1 \end{array}. \label{rk2}
\end{eqnarray}
As this method has $2$ stages, the hypotheses of Theorem \ref{th} are satisfied for $q=2$ and, because of that, the simplifying assumptions apply. Because of that, the full discretization formulas with $p=1$ can be applied and the method reads
\begin{eqnarray}
K_{n,1,h}&=&U_h^n, \nonumber \\
K_{n,2,h}&=&e^{\frac{k}{2} A_{h,0}}U_h^n +\frac{k}{2}\varphi_1(\frac{k}{2} A_{h,0})[F_{n,1,h}+C_h \partial u(t_n)], \nonumber \\
U_h^{n+1}&=&e^{k A_{h,0}}U_h^n+k \varphi_1(k A_{h,0})[F_{n,2,h}+C_h \partial u(t_n)]+k^2 \varphi_2(k A_{h,0}) \partial \dot{u}(t_n). \label{rk2p1}
\end{eqnarray}
This implies that local order $2$ is achieved and, by a summation-by-parts argument, also global order is obtained (see \cite{cm_erk}).

We can also check that (\ref{cond_lambda}) is satisfied. Because of that, the full discretization simplified formulas for $p=2$ can be applied. In this case, this reads
\begin{eqnarray}
K_{n,1,h}&=&U_h^n, \nonumber \\
K_{n,2,h}&=&e^{\frac{k}{2} A_{h,0}}U_h^n+\frac{k}{2} \varphi_1(\frac{k}{2}A_{h,0})[F_{n,1,h}+C_h \partial u(t_n)]+\frac{k^2}{4} \varphi_2(\frac{k}{2}A_{h,0}) C_h \partial \dot{u}(t_n), \nonumber \\
U_h^{n+1}&=& e^{k A_{h,0}}U_h^n+k \varphi_1(k A_{h,0})[F_{n,2,h}+C_h \partial u(t_n)] \nonumber \\
&&+k^2 \varphi_2(k A_{h,0})C_h \partial \big[\dot{u}(t_n)+\frac{k}{2}[f_t(t_n,u(t_n))+f_u(t_n,u(t_n))\dot{u}(t_n)]\big] \nonumber \\
&&+k^3 \varphi_3(k A_{h,0})C_h \partial [\ddot{u}(t_n)-f_t(t_n,u(t_n))-f_u(t_n,u(t_n))\dot{u}(t_n)].
\label{rk2p2}
\end{eqnarray}
As stated in Remark \ref{rem1}, with both Dirichlet and Neumann boundary conditions and $p=1$, all the terms on the boundary can be exactly calculated in terms of data. As for $p=2$, as stated in Remark \ref{rem2}, with Dirichlet boundary condition the required boundaries can be exactly calculated in terms of data and, with Neumann condition, time numerical differentiation for a first derivative is required. For that, we have used a second-order Taylor expansion for the first time step and then, a 2-BDF formula for the next ones.

We have compared the different techniques which have been described here with the standard method of lines, which consists of applying (\ref{rk2}) in the form of (\ref{eerk}) to the space discretization of (\ref{laibvp}), i.e., to
\begin{eqnarray}
\dot{U}_h(t)&=& A_{h,0} U_h(t)+C_h g(t)+U_h.^2 +P_h h(t), \nonumber \\
U_h(0)&=& P_h u_0. \label{waor}
\end{eqnarray}
We can see that procedure (\ref{rk2p2}) is the most efficient for both types of boundary conditions. (See Figures \ref{f1} and \ref{f1dn}, where it is clear that, for a same amount of CPU time, the technique with the smallest error is the one corresponding to the suggested technique with $p=2$. A thorough explanation for that is given in \cite{CR1}).

We have also compared in the same figures with another EERK method in \cite{HO},
\begin{eqnarray}
\begin{array}{c|cc} 0 &  & \\ \frac{1}{2} & \frac{1}{2}\varphi_{1,2} &  \\  \hline & \varphi_1-2 \varphi_2 & 2 \varphi_2 \end{array}, \label{rk2b}
\end{eqnarray}
which happens to have stiff order $2$ for problem (\ref{p1}). (Notice that this method and (\ref{rk2}) have the same underlying RK method.) When implementing this method by integrating directly (\ref{waor}) with the formulas in (\ref{eerk}), order $2$ is observed, but is much less efficient than the procedure in (\ref{rk2p2}).

\subsection{Bidimensional problem}

\begin{figure}
\begin{center}
\vspace{-4cm}
\includegraphics[width=130mm]{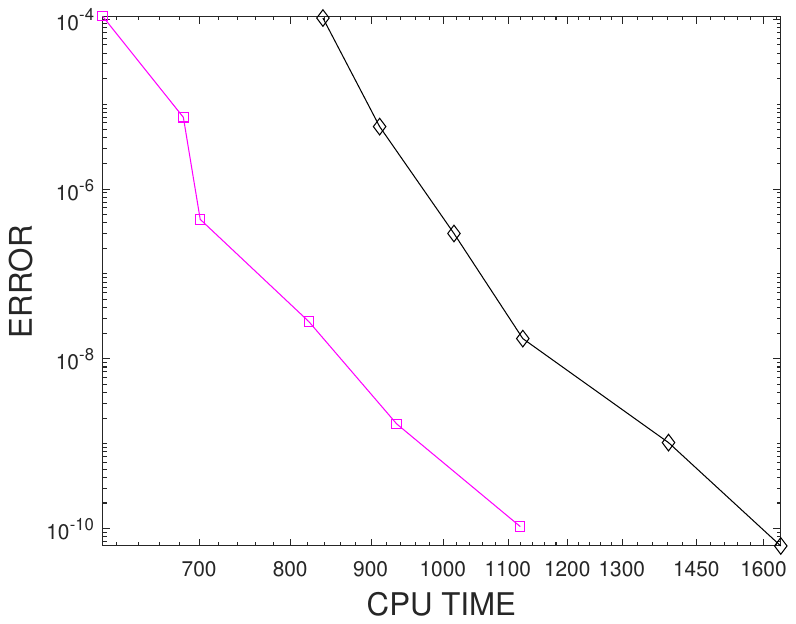}
\vspace{-5cm}
\end{center}
\caption{Error against CPU time when integrating problem (\ref{p2}), using (\ref{krogstad}) with the suggested technique (\ref{fdk3})-(\ref{fdu3}) with $p=3$ (magenta squares) and the method in \cite{LO} with stiff order $4$ when applied directly to (\ref{sdd}) (black diamonds).}
\label{f2}
\end{figure}

We have also considered the bidimensional problem
\begin{eqnarray}
u_t=u_{xx}+u_{yy}+u^2+h(x,y,t), \quad x,y\in [0,1], \quad t\in [0,1], \label{p2}
\end{eqnarray}
with function $h$ and initial and Dirichlet boundary conditions such that the exact solution is $u(x,y,t)=\cos(t+x+y)$.
As space discretization, we have taken the $9$-point formula \cite{S}, which determines certain matrices $A_{h,0}$, $C_h$ and $D_h$ in (\ref{sd}).  For the time integration, we have considered Krogstad's method \cite{K}
\begin{eqnarray}
\begin{array}{c|cccc} 0 &  &  & & \\ \frac{1}{2} & \frac{1}{2}\varphi_{1,2} & & & \\ \frac{1}{2} & \frac{1}{2}\varphi_{1,3}-\varphi_{2,3}& \varphi_{2,3} &  & \\ 1 & \varphi_{1,4}-2 \varphi_{2,4}& 0 & 2\varphi_{2,4} & \\
 \hline & \varphi_1-3 \varphi_2+4 \varphi_3 & 2\varphi_2-4 \varphi_3& 2\varphi_2-4 \varphi_3 & 4\varphi_3-\varphi_2 \end{array}, \label{krogstad}
\end{eqnarray}
which has classical order $4$ and $4$ stages. Because of that, Theorem \ref{th} can be applied and therefore (\ref{cond_mu}) follows. We also notice that condition (\ref{cond_lambda}) is satisfied. We have then applied formulas (\ref{fdk3})-(\ref{fdu3}) to implement the suggested technique to avoid order reduction with $p=3$. In such a way, we get both local and global order $4$. We notice that many terms cancel in these formulas since $\lambda_{i,j,l}=0$ for $l=3,4$ and $\mu_{i,4}=0$. We have again used Krylov methods \cite{nw} to calculate the exponential operators (We notice that, as the matrix $A_{h,0}$ is not so sparse in this case because it is just the product of the inverse of a sparse mass matrix times another sparse matrix, a modification of the subroutines in \cite{nw} has been used following the lines of Section 7.1 in \cite{GG}, so as to take profit of the highly sparse structure of the underlying matrices.) As stated in Section \ref{sect23}, no numerical differentiation is required for the stages. However, as stated in Remark \ref{rem3}, for the approximation of $\partial A(f_t+f_u \dot{u})$ in $U_h^{n+1}$, numerical differentiation is required. In a first place,  $u_x$ and $u_y$ are required at the boundary, for which we have used a 4th-order BDF formula in the sides of the square in which we could not calculate it directly in terms of $g$. In a second place,  $\dot{u}_x$, $\dot{u}_y$ at the boundary are also required. We have approximated them using $4$th-order BDF formulas.

We have also made a comparison in CPU time of the suggested tecnique using $p=3$ with another method which shows order 4 without avoiding order reduction. We firstly tested with the $5$-stages method (5.19) in \cite{HO} because it was suggested there as a $4$th-order method for stiff problems with vanishing boundary conditions, but it happened to show just order $3$ with non-vanishing boundary values. Because of that, we decided to consider the $5$th-order method for stiff problems and vanishing boundary conditions which is constructed in \cite{LO}. This method is $4$th-order accurate when applied directly to the space discretized system
\begin{eqnarray}
 \dot{U}_h(t)&=&A_{h,0} U_h(t)+C_h g(t)+U_h.^2+P_h h(t)+D_h[(g(t)^2+\partial h(t)-\dot{g}(t)], \nonumber \\
 U_h(0)&=&P_h u(0), \label{sdd}
\end{eqnarray}
and has $8$ stages.
The comparison in CPU time is shown in Figure \ref{f2}, where it is observed that the latter is about $1/2$ more expensive than the suggested modification of Krogstad's method when avoiding order reduction.

\noindent {\bf Acknowledgments.} This research has been supported by  Junta de Castilla y Le\'on and Feder through project VA169P20.

\end{document}